\documentclass[journal]{IEEEtran}
\ifCLASSINFOpdf
\else
\fi
\usepackage{graphicx}
\usepackage{amssymb}
\usepackage{amsmath}
\usepackage{amsthm}
\usepackage{relsize}
\newtheorem{thm}{Theorem}
\newtheorem{prop}{Proposition}
\newtheorem{rem}{Remark}

\begin{document}
\title{A geometrical look at MOSPA Estimation using Transportation Theory}


\author{Gabriel~M.~Lipsa,
	Marco~Guerriero,~\IEEEmembership{Senior,~IEEE,}
	\thanks{G. Lipsa and M. Guerriero are with the General Electric Global Research Center, Niskayuna, NY, USA,
		e-mails: (gabriel.lipsa,marco.guerriero)@ge.com}% <-this % stops a space
	\thanks{Manuscript received xxx, 2016.}}


%Telephone: (800) 555--1212, Fax: (888) 555--1212}


\maketitle

\begin{abstract}
It was shown in \cite{hoffman2004multitarget} that the Wasserstein distance is equivalent to the Mean Optimal
Sub-Pattern Assignment (MOSPA) measure for empirical probability density functions. A more recent paper \cite{baum2015wasserstein}, extends on it by drawing new connections between the MOSPA concept, which is getting a foothold in the multi-target tracking community, and the Wasserstein distance, a metric widely used in theoretical statistics. However, the comparison between the two concepts has been overlooked. In this letter we prove that the equivalence of Wasserstein distance with the MOSPA measure holds for general types of probability density function. This non trivial result allows us to leverage one recent finding in the computational geometry literature to show that the Minimum MOPSA (MMOSPA) estimates are the centroids of additive weighted Voronoi regions with a specific choice of the weights.
\end{abstract}

% \fi
\IEEEpeerreviewmaketitle

\section{Introduction}
The mean squared error (MSE) has long been the dominant quantitative performance metric in the field
of signal processing. An estimator which minimizes the MSE is referred to as a minimum MSE (MMSE)
estimator. In target tracking \cite{bar2011tracking}, the traditional problem is posed as the finding of the MMSE estimate of the target states. Since the MMSE estimate is given by the expected value of the posterior probability density function, which intrinsically has an ordering (labeling) of the states, the MMSE estimator can be classified as a labeled estimator.
In some applications the labeling of the objects is not relevant. For these problems, it
is more reasonable to instead of minimizing the MSE try to minimize a measure which eschews target
labeling.
A measure that has received an increasing amount of attention in the later years, and which can be seen
as a natural extension of the MSE to label-free estimation, is the OSPA
metric \cite{schuhmacher2008consistent}. The OSPA is a label-free correspondent to the squared error. The MOSPA,
which is the counterpart of the MSE, was introduced in \cite{guerriero2010shooting} where the authors also described how to calculate the MMOSPA estimates. Explicit solutions for MMOSPA estimation are only available in the scalar case \cite{crouse2011aspects}. In \cite{crouse_advances2013} and references therein, various techniques for approximating MMOSPA estimates are presented .
In \cite{baum2015wasserstein}, a connection between the empirical MMOSPA estimate and the Wasserstein barycenter for point cloud was established. This result builds upon the Lemma 1 in \cite{baum2015wasserstein} which states that the Wasserstein distance coincides with the MOSPA for empirical probability densities defined on sets with the same cardinality.
The Wasserstein distance defined using empirical probability densities can be computed solving a linear programming (LP) problem \cite{hoffman2004multitarget,rubner2000earth}. The LP formulation of the transportation problem, is also known as the Hitchcock-Koopmans transportation problem \cite{Rao2009}.
The results of this letter are twofold:
\begin{itemize}
	\item We extend the result in \cite{baum2015wasserstein,hoffman2004multitarget} for a wider class of probability measures,  to draw a theoretical connection between the MOSPA measure and the more general transportation problem, known as Monge-Kantorovich transportation problem \cite{Villani, Kantorovich1948, Monge1781}.
	\item This main finding, in conjunction with a recent result in computational geometry \cite{Geiß2012}, allows us to provide new insights on MOSPA estimation revealing interesting geometrical structure and properties of the MMOSPA estimates. 
\end{itemize}
The remainder of the paper is organized as follows. In Section \ref{sec:probform} we formalize our problem. Section \ref{sec:mainresults} contains
our main theoretical contributions and also provides a geometrical interpretation of the MMOSPA and in Section \ref{sec:conclusions} we summarize our conclusions.

\section{Problem Formulation}
\label{sec:probform}
In this section, we will present the main problem of this letter. We will first define notations used in this paper and then we will define notions of interest such as OSPA, MOSPA, MMOSPA and the Wasserstein distance.

Let us assume that there are $N$ objects of interest, which reside in the space $\mathbb{R}^{n_x}$ with $n_x$ being a positive integer. The states of all the objects are denoted by the sequence of vectors $\{\mathbf{X}_{i} \}_{i=1}^{N} \in \mathbb{R}^{n_x}$. The vectors $\{\mathbf{X}_{i} \}_{i=1}^{N}$ are stochastic with a joint probability measure $\mu$\footnote{Any practical multi-target tracking setting would require measurements from which the target states estimates are computed. If we denote the measurements by $\mathbf{Z}$, $\mu\left(\mathbf{X} \right)$, which corresponds to the joint posterior measure (i.e. a typical assumption in target tracking is to use a Gaussian Mixture model to represent the joint distribution), should be replaced by $\mu\left(\mathbf{X}|\mathbf{Z} \right)$. However, for clarity purposes we will use $\mu$ without the conditions on the measurement.} defined on $\mathbb{R}^{N\times n_x}$. Moreover we assume that $\mu$ is absolutely continuous with respect to the Lebesque measure \cite{Billingsley}.

Define the stacked vector $\mathbf{X}$ as follows \footnote{The symbol T stands for transpose.}:
\begin{equation}
\label{stackedvector}
\mathbf{X} = \left[\mathbf{X}_1^{T} \quad \mathbf{X}_2^{T} \quad \ldots \quad \mathbf{X}_N^{T}   \right]^{T}
\end{equation}    
For a sequence of vectors of states estimates $\{\mathbf{\hat{X}}_{i} \}_{i=1}^{N}$, let us define the stacked vector of states estimates $\mathbf{\hat{X}}$ as in equation~(\ref{stackedvector}). %Moreover, assume for identifiability that $\mathbf{\hat{X}}_{i} \neq \mathbf{\hat{X}}_{j}, \forall i\neq j$. %This assumption is not important and it can be dropped, but it is used in order to make the presentation of this paper more clear. 

Define $\Pi_{N}$ to be the set of permutations on the set $\{ 1, 2, \ldots, N\}$. 
%Let us use the notation $\pi_i$ to indicate a permutation in $\Pi_{N}$, with $i \in \{ 1, 2, \ldots, N!\}$.%
For a permutation $\pi \in \Pi_N$\footnote{A permutation $\pi \in \Pi_N$ is a bijective mapping from the set $\{ 1, 2, \ldots, N\}$ to the set $\{ 1, 2, \ldots, N\}$. The value of the mapping for a particular index $i$ is denoted by $\pi(i)$.} and a stacked vector defined in equation~(\ref{stackedvector}), let us define $\pi(\mathbf{X})$ as follows:
\begin{equation}
\pi(\mathbf{X}) =  \left[\mathbf{X}_{\pi(1)}^{T} \quad \mathbf{X}_{\pi(2)}^{T} \quad \ldots \quad \mathbf{X}_{\pi(N)}^{T}   \right]^{T}
\end{equation}
The vector $\pi(\mathbf{X})$ permutes the single objects states in $\mathbf{X}$ according to $\pi$.
Define OSPA \cite{baum2015polynomial} as follows \footnote{In this letter, we use the definition of OSPA for sets with the same cardinality.}:
\begin{equation}
\label{OSPA_ver1}
d^{OSPA}(\mathbf{X},\mathbf{\hat{X}}) = \min_{\pi \in \Pi_N}\sum_{i=1}^{N}\left(\mathbf{X}_i - \mathbf{\hat{X}}_{\pi(i)} \right)^{T}\left(\mathbf{X}_i - \mathbf{\hat{X}}_{\pi(i)} \right)
\end{equation}

Let us define OSPA using the stacked notation as follows:
\begin{equation}
\label{OSPA_ver2}
d^{OSPA}(\mathbf{X},\mathbf{\hat{X}}) = \min_{\pi \in \Pi_N}\left(\mathbf{X} - \pi\left(\mathbf{\hat{X}}\right) \right)^{T}\left(\mathbf{X} - \pi\left(\mathbf{\hat{X}}\right)\right)
\end{equation}

Let us define MOSPA and the relative MMOSPA as follows  \cite{guerriero2010shooting} \footnote{The notation $E_{\mu}$ denotes expectation with respect to the measure $\mu$} :
\begin{equation}
\label{MOSPA}
MOSPA\left(\mu, \mathbf{\hat{X}}\right) = E_{\mu}\left[d^{OSPA}(\mathbf{X},\mathbf{\hat{X}}) \right]
\end{equation}
\begin{equation}
\label{MMOSPA}
  \mathbf{\hat{X}}^{MMOSPA} =\textrm{arg}\min_{\mathbf{\hat{X}}} E_{\mu}\left[d^{OSPA}(\mathbf{X},\mathbf{\hat{X}}) \right]
\end{equation}

Let us define the $p^{\textrm{th}}$ Wasserstein distance \cite{Villani}, \cite{Ambrosio} between two probability measures on some space $M$ as:
\begin{equation}
\label{wasser}
W_{p} (\nu_1, \nu_2):=\left( \inf_{\gamma \in \Gamma (\nu_1, \nu_2)} \int_{M \times M} d(x, y)^{p} \, \mathrm{d} \gamma (x, y) \right)^{1/p}
\end{equation}
$\Gamma (\nu_1, \nu_2)$ denotes the set of all joint measures on $M \times M$ with marginal measures $\nu_1$ and $\nu_2$.
For this paper we use the Euclidean distance for $d$, $p=2$ and $M$ is the space $\mathbb{R}^{N\times n_x}$.

We are now ready to formulate the main problem of this letter. It will be shown that the following equation holds, which connects the MOSPA and the Wasserstein distance:
\begin{equation}
\label{eq_prob_form}
MOSPA\left(\mu, \mathbf{\hat{X}}\right) = W^{2}_{2} (\mu, \nu)
\end{equation}
where $\nu$ is a discrete measure which depends on $\mathbf{\hat{X}}$ and it will be defined later.

Let us define the collection of sets $\mathbb{S}_{\pi}\left(\mathbf{\hat{X}}\right)$ \footnote{Hereafter, we use the short notation $\mathbb{S}_{\pi}$ for $\mathbb{S}_{\pi}\left(\mathbf{\hat{X}}\right)$.} for all $\pi \in \Pi_N$ as follows:
\begin{equation}
\label{S_pi}
\begin{split}
\mathbb{S}_{\pi}\left(\mathbf{\hat{X}}\right) = \{\mathbf{X} \in \mathbb{R}^{N\times n_x}: \left(\mathbf{X} - \pi\left(\mathbf{\hat{X}}\right) \right)^{T}\left(\mathbf{X} - \pi\left(\mathbf{\hat{X}}\right)\right)  \\ \leq \left(\mathbf{X} - \tilde{\pi}\left(\mathbf{\hat{X}}\right) \right)^{T}\left(\mathbf{X} - \tilde{\pi}\left(\mathbf{\hat{X}}\right)\right), \forall \tilde{\pi} \in \Pi_N \}
\end{split}
\end{equation}	

It follows then, that for any two different permutations $\pi$ and $\tilde{\pi}$, the set $\mathbb{S}_{\pi} \cap \mathbb{S}_{\tilde{\pi}}$ has Lebesque measure zero, hence its measure with respect to $\mu$ is also zero. Here we assume without loss of generality that $\mathbf{\hat{X}}_{i} \neq \mathbf{\hat{X}}_{j}, \forall i\neq j$.

For a fixed deterministic $\mathbf{\hat{X}}$, let us define the discrete random variable (d.r.v.) $\mathbf{Y}$ as follows \footnote{ The symbol $\mathbb{I}_{\mathbb{S}}$ denotes the indicator function of the set $\mathbb{S}$, i.e. $\mathbb{I}_{\mathbb{S}}\left(x \right)= 1$ if $x \in \mathbb{S}$ and zero otherwise.}:
\begin{equation}
\label{Yrandom}
\mathbf{Y} = \sum_{\pi \in \Pi_N} \pi\left(\mathbf{\hat{X}}\right) \cdot \mathbb{I}_{\mathbb{S}_{\pi}}\left(\mathbf{X}\right)
\end{equation}
From equation~(\ref{Yrandom}) we conclude that $\mathbf{Y}$ takes value $\pi\left(\mathbf{\hat{X}}\right)$ with probability $\mu(\mathbb{S}_{\pi})$ for all $\pi \in \Pi_N$.
Then, the d.r.v. $\mathbf{Y}$ induces the discrete probability measure $\nu$ on  $\mathbb{R}^{N\times n_x}$ which is defined as follows:
\begin{equation}
\label{nu_label}
\nu(x) = \sum_{\pi \in \Pi_N} \mu(\mathbb{S}_{\pi}) \cdot \delta \left(x-\pi\left(\mathbf{\hat{X}}\right)\right)
\end{equation}
The measure $\nu$, which depends on $\mathbf{\hat{X}}$, will play the role of measure $\nu$ from equation~(\ref{eq_prob_form}).

\section{Main Results}
\label{sec:mainresults}

In this section, we will present the main results of this letter. We will formulate and prove Theorem~\ref{thm1}, which shows the connection between the Wasserstein distance and the MOSPA for general measures and then we will prove certain geometrical properties of the optimal MOSPA.

\subsection{MOSPA meets Wasserstein in the general case}

In this subsection, we will formulate and prove the main theorem of the paper, in which we establish the connection between the Wasserstein distance and the MOSPA.

\begin{thm}
\label{thm1}
Given a probability measure $\mu$ defined on $\mathbb{R}^{N\times n_x}$, a vector $\hat{X} \in \mathbb{R}^{N\times n_x}$ and a probability measure $\nu$ defined in equation~(\ref{nu_label}), the following holds: 
\begin{equation}
\label{thm1_eq}
 E_{\mu}\left[d^{OSPA}(\mathbf{X},\mathbf{\hat{X}}) \right] = W^{2}_{2} (\mu, \nu)
\end{equation}
\end{thm}

\begin{proof}
	We first show that 
\begin{equation}
	 E_{\mu}\left[ \left(\mathbf{X} - \mathbf{Y} \right)^{T}\left(\mathbf{X} - \mathbf{Y}\right)\right] = E_{\mu}\left[d^{OSPA}(\mathbf{X},\mathbf{\hat{X}}) \right]
\end{equation}
with $\mathbf{Y}$ defined in equation~(\ref{Yrandom}).

\begin{equation*}
\begin{split}
E_{\mu} &\left[ \left(\mathbf{X} - \mathbf{Y} \right)^{T}\left(\mathbf{X} - \mathbf{Y}\right)\right]  \\
& = \int_{\mathbb{R}^{N\times n_x}} \left(\mathbf{X} - \mathbf{Y} \right)^{T}\left(\mathbf{X} -\mathbf{Y} \right) d\mu \\
&=\sum_{\pi \in \Pi_N} \int_{\mathbb{S}_{\pi}} \left(\mathbf{X} - \mathbf{Y} \right)^{T}\left(\mathbf{X} -\mathbf{Y} \right) d\mu\\
&\stackrel{a}{=}\sum_{\pi \in \Pi_N} \int_{\mathbb{S}_{\pi}} \left(\mathbf{X} - \sum_{\pi \in \Pi_N} \pi\left(\mathbf{\hat{X}}\right) \cdot \mathbb{I}_{\mathbb{S}_{\pi}}\left(\mathbf{X}\right) \right)^{T} \\
& \quad \quad \quad \quad \quad \left(\mathbf{X} -\sum_{\pi \in \Pi_N} \pi\left(\mathbf{\hat{X}}\right) \cdot \mathbb{I}_{\mathbb{S}_{\pi}}\left(\mathbf{X}\right) \right) d\mu \\
&\stackrel{b}{=} \sum_{\pi \in \Pi_N} \int_{\mathbb{S}_{\pi}} \left(\mathbf{X} - \pi\left(\mathbf{\hat{X}}\right) \right)^{T}  \left(\mathbf{X} - \pi\left(\mathbf{\hat{X}}\right) \right) d\mu \\
&\stackrel{c}{=}E_{\mu}\left[d^{OSPA}(\mathbf{X},\mathbf{\hat{X}}) \right]
\end{split}
\end{equation*}
The equalities ($a$) and ($b$) follow from the definition of $\mathbf{Y}$ in equation~(\ref{Yrandom}) and the equality 
($c$) follows from the definition of $\mathbb{S}_{\pi}$ in equation~(\ref{S_pi}) and the definition of OSPA in equation~(\ref{OSPA_ver2}).
Next we show that 
\begin{equation}
\label{ineq1}
	 W^{2}_{2} (\mu, \nu) \leq E_{\mu} \left[ \left(\mathbf{X} - \mathbf{Y} \right)^{T}\left(\mathbf{X} - \mathbf{Y}\right)\right]  
\end{equation}
Equation ~(\ref{Yrandom}) defines the probability measure $\nu$ and moreover defines a joint measure  $\gamma \in \Gamma(\mu,\nu)$. Hence, from the definition of the Wasserstein distance in equation~(\ref{wasser}), equation ~(\ref{ineq1}) immediately follows.
We show the reverse inequality next that
\begin{equation}
\label{ineq2}
W^{2}_{2} (\mu, \nu) \geq E_{\mu} \left[ \left(\mathbf{X} - \mathbf{Y} \right)^{T}\left(\mathbf{X} - \mathbf{Y}\right)\right]  
\end{equation}
Choose an arbitrary joint probability measure $\gamma \in \Gamma(\mu,\nu)$. From $\gamma$ we can define the conditional probability measure $\gamma_{\nu|\mu}$ with respect to the measure $\mu$. It follows then, that $\gamma_{\nu|\mu}$ is a discrete probability measure with the same support as $\nu$. 
We can write the following:
\begin{equation}
\label{ineq2_tmp}
\begin{split}
 &E_{\gamma} \left[ \left(\mathbf{X} - \mathbf{Y} \right)^{T}\left(\mathbf{X} - \mathbf{Y}\right)\right] \\
 &= \int_{\mathbb{R}^{N\times n_x} \times \mathbb{R}^{N\times n_x}} \left(\mathbf{X} - \mathbf{Y} \right)^{T}\left(\mathbf{X} - \mathbf{Y}\right) d\gamma \\
 &= \int_{\mathbb{R}^{N\times n_x}} \int_{\mathbb{R}^{N\times n_x}}  \left(\mathbf{X} - \mathbf{Y} \right)^{T}\left(\mathbf{X} - \mathbf{Y}\right) d\gamma_{\nu|\mu} d\mu \\
 &= \sum_{\pi \in \Pi_N}\int_{\mathbb{S}_{\pi}} \int_{\mathbb{R}^{N\times n_x}}  \left(\mathbf{X} - \mathbf{Y} \right)^{T}\left(\mathbf{X} - \mathbf{Y}\right) d\gamma_{\nu|\mu} d\mu\\
 &\stackrel{a}{\geq}\sum_{\pi \in \Pi_N}\int_{\mathbb{S}_{\pi}} \int_{\mathbb{R}^{N\times n_x}}  \left(\mathbf{X} -  \pi\left(\mathbf{\hat{X}}\right) \right)^{T}\left(\mathbf{X} -  \pi\left(\mathbf{\hat{X}}\right)\right) d\gamma_{\nu|\mu} d\mu\\
 &\stackrel{b}{=}\sum_{\pi \in \Pi_N}\int_{\mathbb{S}_{\pi}}  \left(\mathbf{X} -  \pi\left(\mathbf{\hat{X}}\right) \right)^{T}\left(\mathbf{X} -  \pi\left(\mathbf{\hat{X}}\right)\right) \int_{\mathbb{R}^{N\times n_x}}d\gamma_{\nu|\mu}d\mu\\
 &=\sum_{\pi \in \Pi_N}\int_{\mathbb{S}_{\pi}} \left(\mathbf{X} -  \pi\left(\mathbf{\hat{X}}\right) \right)^{T}\left(\mathbf{X} -  \pi\left(\mathbf{\hat{X}}\right)\right) d\mu \\
 &\stackrel{c}{=}E_{\mu}\left[d^{OSPA}(\mathbf{X},\mathbf{\hat{X}}) \right]
\end{split}
\end{equation}
The inequality ($a$) follows from the definition of  $\mathbb{S}_{\pi}$ in equation~(\ref{S_pi}), equality ($b$) follows from the fact that $\pi\left(\mathbf{\hat{X}}\right)$ is a deterministic vector for $\mathbf{X} \in \mathbb{S}_{\pi}$, which does not depend on the conditional probability measure $\gamma_{\nu|\mu}$. The equality 
($c$) follows from the definition of $\mathbb{S}_{\pi}$ in equation~(\ref{S_pi}) and the definition of OSPA in equation~(\ref{OSPA_ver2}).
From equation~(\ref{ineq2_tmp}), we conclude that 
\begin{equation}
 E_{\gamma} \left[ \left(\mathbf{X} - \mathbf{Y} \right)^{T}\left(\mathbf{X} - \mathbf{Y}\right)\right] \geq
 E_{\mu}\left[d^{OSPA}(\mathbf{X},\mathbf{\hat{X}}) \right]
\end{equation}
Taking the infimum over all  $\gamma \in \Gamma(\mu,\nu)$, equation~(\ref{ineq2}) follows.
Hence, from equation~(\ref{ineq1}) and equation~(\ref{ineq2}), equation ~(\ref{thm1_eq}) follows.
\end{proof}

\subsection{Geometry of the optimal MMOSPA}
In this subsection, we will prove that the geometry of the optimal partition from Theorem~\ref{thm1} satisfies certain geometric properties. It was shown in Theorem~(\ref{thm1}), that:
\begin{equation}
E_{\mu}\left[d^{OSPA}(\mathbf{X},\mathbf{\hat{X}}) \right] = W^{2}_{2} (\mu, \nu)
\end{equation}

with the measure $\nu$ being a discrete probability measure, which takes values $\pi\left(\mathbf{\hat{X}}\right)$ with probability $\mu\left(\mathcal{S}_{\pi}\right)$. Theorem 2 from \cite{Geiß2012} shows that, the optimal partition of the space $\mathcal{R}^{N\times n_{x}}$ which achieves $W^{2}_{2} (\mu, \nu)$ exists, it is unique and it is given by the additive weighted Voronoi regions defined as follows:
\begin{equation}
\begin{split}
\mathcal{C}_{i}\left(\mathcal{P}, \Omega\right) =\{\mathbf{X} \in \mathcal{R}^{N\times n_{x}}: &\left(\mathbf{X} - p_{i}\right)^{T}\left(\mathbf{X} - p_{i}\right) + \omega_{i}\leq\\
& \left(\mathbf{X} - p_{j}\right)^{T}\left(\mathbf{X} - p_{j}\right) + \omega_{j} ,\\
&\forall j \in \{1, 2, \ldots, N!\} \}
\end{split}
\end{equation}
where $\mathcal{P} = \{p_1, \ldots, p_{N!}\} \subseteq \mathcal{R}^{N\times n_{x}}$ is a set of centroids and $\Omega = \{\omega_1, \ldots, \omega_{N!}\} \subseteq \mathcal{R}$ is a set of real numbers. 

\begin{prop}
	\label{prop1}
$\pi\left(\mathbf{\hat{X}}\right)$ are the centroids of the additive weighted Voronoi regions $\mathcal{C}_{i}\left(\mathcal{P}, \Omega\right)$, with $\omega_{i} = \omega_{j}, \forall i \neq j$.
\end{prop}
\begin{proof}
For a given $\mathbf{\hat{X}}$, it follows from Theorem 2 in~\cite{Geiß2012}, that $\mathcal{C}_{i}\left(\mathcal{P}, \Omega\right)$ are optimal sets. Moreover from the proof of Theorem 1, it follows that the sets $\mathcal{S}_{\pi}$ are optimal. Then it can be seen that the sets $\mathcal{S}_{\pi}$ are the same as the sets $\mathcal{C}_{i}\left(\mathcal{P}, \Omega\right)$ with $\omega_{i} = \omega_{j}, \forall i \neq j$.
\end{proof}	
	
\begin{figure}
	\centering
		\includegraphics[scale=0.5]{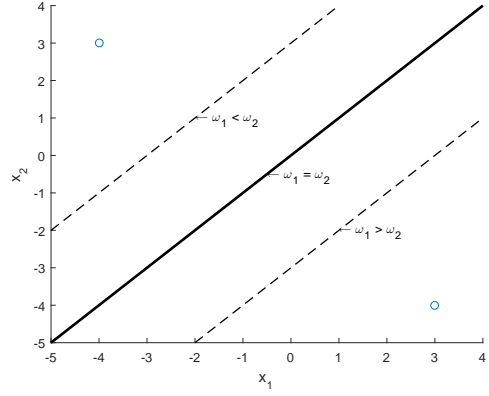}
	\caption{Voronoi diagrams for the two targets case. The circles represent the two unlabeled target state estimates.}\label{VoronoiFigure}
\end{figure}	

Interestingly, the results above are a dual of the results from \cite{guerriero2010shooting} with respect to the labeling of the target states versus target states estimates.
In \cite{guerriero2010shooting}, the MMOSPA estimates were fixed while $\mu$ was folded. In this letter, the measure $\mu$ remains fixed while the MMOSPA estimates define the probability measure $\nu$, which depends on the permutations of the MMOSPA estimates themselves and on the Voronoi diagrams from Proposition \ref{prop1}.

Figure~\ref{VoronoiFigure} shows an example of the Voronoi diagram for MMOSPA estimates. In this case $N=2$ and $n_x = 1$, i.e. there are two one-dimensional targets. Moreover, let the sequence of target state estimates be the set $\{-4,3\}$. Then, for a realization of the vector $\mathbf{X}$ above the solid line in Figure~\ref{VoronoiFigure}, the target states estimates are represented by the vector $\left[-4 \quad 3 \right]$, while for a realization of the vector $\mathbf{X}$ below the solid line,  the target states estimates are represented by the vector $\left[3 \quad -4 \right]$. The dotted lines show examples of additive Voronoi diagrams with different weights. For the case with three one-dimensional targets, there will be  six different possible target approximations and six regions separated by hyperplanes in the three dimensional Euclidean space.

\begin{rem}
	\label{Remark1}
	The results in \cite{Geiß2012} and in general in transportation theory \cite{Villani} hold for general distances than Euclidean and the geometrical properties from Proposition~\ref{prop1} remain still valid. 
	
	For example, in the case of a more general distance GOSPA (General OSPA)\footnote{This new distance might be useful in applications where objects labels are more "important" than others. This selective importance is reflected in the definition of the new distance which will ``favor" some permutations over others for the MOSPA estimation task.}:
	\begin{equation}
	d^{GOSPA}\left(\mathbf{X},\mathbf{\hat{X}} \right) = \left(\mathbf{X} - \pi\left(\mathbf{\hat{X}}\right) \right)^{T}\mathbf{Q}\left(\mathbf{X} - \pi\left(\mathbf{\hat{X}}\right)\right)
	\end{equation}
	 where $\mathbf{Q}$ is a positive definite matrix, the additively weigthed Voronoi diagrams are no longer symmetric (i.e. $\omega_{i} \neq \omega_{j}) $\footnote{Similarly to the hyperplane in Figure \ref{VoronoiFigure}, the new separating hyperplane will be closer to one of the two MMOSPA estimates (circles in the figure).}.
	Lastly, if more general distances are used \cite{Geiß2012}, the Voronoi diagrams will be separated by more general manifolds than hyperplanes \cite{Villani}.
\end{rem}

\section{Conclusion}
\label{sec:conclusions}

The main result of this letter establishes the equivalence between the MOSPA measure, which is a concept widely used in target tracking community, and Wasserstein distance between one continuous measure and one discrete measure. This finding allowed us to draw a connection with a recent result in computational geometry \cite{Geiß2012}, which showed that additively weigthed Voronoi diagrams can optimally solve some cases of the Monge-Kantorovich transportation problem, with one measure being discrete.
More specifically, we were able to show that MMOSPA estimates are exactly equal to the centroid of these Voronoi diagrams for a particular choice of the weights. Revealing geometrical structures for the MMOSPA estimates advances our understanding of the MOSPA estimation problem drawing upon different scientific fields.
In the future we are planning to extend the current results, if possible, to the more general case of MOSPA measure defined for sets with different cardinalities.


\section*{Acknowledgment}

The authors would like to thank Jayakrishnan Unnikrishnan, Michael Lexa and Peter Spaeth for their comments.



%\bibliographystyle{IEEEtran}

%\bibliographystyle{IEEEtran}
%\bibliography{MMOSPA_Transportation}
%

\end{document}